\documentclass[oneside,reqno,american]{amsart}
\usepackage[T1]{fontenc}
\usepackage[utf8]{inputenc}
\usepackage{xcolor}
\usepackage{babel}
\usepackage{prettyref}
\usepackage{mathrsfs}
\usepackage{mathtools}
\usepackage{amstext}
\usepackage{amsthm}
\usepackage{amssymb}
\usepackage[pdfusetitle,
 bookmarks=true,bookmarksnumbered=false,bookmarksopen=false,
 breaklinks=false,pdfborder={0 0 0},pdfborderstyle={},backref=false,colorlinks=false]
 {hyperref}
\hypersetup{
 colorlinks=true,citecolor=blue,linkcolor=blue,linktocpage=true}

\makeatletter
\numberwithin{equation}{section}
\numberwithin{figure}{section}

\newrefformat{cor}{Corollary~\ref{#1}}
\newrefformat{subsec}{Section~\ref{#1}}
\newrefformat{lem}{Lemma~\ref{#1}}
\newrefformat{thm}{Theorem~\ref{#1}}
\newrefformat{sec}{Section~\ref{#1}}
\newrefformat{chap}{Chapter~\ref{#1}}
\newrefformat{prop}{Proposition~\ref{#1}}
\newrefformat{exa}{Example~\ref{#1}}
\newrefformat{tab}{Table~\ref{#1}}
\newrefformat{rem}{Remark~\ref{#1}}
\newrefformat{def}{Definition~\ref{#1}}
\newrefformat{fig}{Figure~\ref{#1}}
\newrefformat{claim}{Claim~\ref{#1}}

\makeatother

\theoremstyle{plain}
\newtheorem{thm}{\protect\theoremname}[section]
\theoremstyle{definition}
\newtheorem{defn}[thm]{\protect\definitionname}
\theoremstyle{plain}
\newtheorem{lem}[thm]{\protect\lemmaname}
\newtheorem{cor}[thm]{\protect\corollaryname}
\theoremstyle{remark}
\newtheorem{rem}[thm]{\protect\remarkname}
\theoremstyle{plain}
\newtheorem{prop}[thm]{\protect\propositionname}
\newrefformat{cor}{Corollary \ref{#1}}
\newrefformat{def}{Definition \ref{#1}}
\newrefformat{lem}{Lemma \ref{#1}}
\newrefformat{prop}{Proposition \ref{#1}}
\newrefformat{rem}{Remark \ref{#1}}
\newrefformat{sec}{Section~\ref{#1}}
\newrefformat{subsec}{Section~\ref{#1}}
\newrefformat{thm}{Theorem \ref{#1}}
\providecommand{\corollaryname}{Corollary}
\providecommand{\definitionname}{Definition}
\providecommand{\lemmaname}{Lemma}
\providecommand{\propositionname}{Proposition}
\providecommand{\remarkname}{Remark}
\providecommand{\theoremname}{Theorem}

\begin{document}
\title{Random Operator-Valued Frames in Hilbert Spaces}
\begin{abstract}
We study strongly measurable random bounded operators on separable
Hilbert spaces and analyze two simple iterations driven by independent
random positive contractions. The first, a Kaczmarz-like iteration,
converges in mean square and almost surely and produces a random operator-valued
frame. In the projection case it yields a Parseval identity. The second,
a residual-weighted iteration, enjoys an exact step-by-step identity:
the accumulated analysis terms plus a residual equal the identity
operator. Under a mild mean-coercivity condition, the residual shrinks
at a geometric rate in expectation, vanishes almost surely, and admits
nonasymptotic tail bounds. As a result, the construction delivers
an almost-sure Parseval frame for any independent sequence of positive
contractions, not only projections. 
\end{abstract}

\author{James Tian}
\address{Mathematical Reviews, 535 W. William St, Suite 210, Ann Arbor, MI
48103, USA}
\email{jft@ams.org}

\maketitle
\tableofcontents{}

\section{Introduction}

Frames extend the idea of orthonormal bases in Hilbert spaces by allowing
redundancy. This builtin redundancy makes them stable and robust to
noise, erasures, and other errors. Because of these features, frames
have become an important tool in signal processing, image analysis,
coding theory, and wireless communications. The theory of frames began
with the work of Duffin and Schaeffer \cite{MR47179} and has since
developed in many directions; see \cite{MR3495345} or \cite{MR1757401}
for surveys.

Two important extensions are fusion frames and operator-valued frames.
Fusion frames, introduced by Casazza et al. \cite{MR2419707,MR2964018},
model distributed sensing and reconstruction using families of subspaces
with weights, and they have applications in coding theory, wireless
communication, and parallel processing. Operator-valued frames, developed
by Casazza, Han, and Larson and studied further by Han and Larson
\cite{MR1686653} and Sun \cite{MR2239250}, go further by replacing
vectors or subspaces with bounded operators. This connects frame theory
with $C^{*}$-algebras, dilation theory \cite{MR275190,MR1976867},
and noncommutative harmonic analysis. Operator-valued frames also
appear in quantum information theory, for example in quantum state
tomography. For more recent developments there is a large and growing
literature covering many directions, which will not be surveyed here.

Randomization has become another fruitful direction. Random frames
and stochastic operator families have been studied for both theoretical
interest and applications in compressed sensing, randomized numerical
linear algebra, and stochastic signal analysis \cite{MR3958700,MR3263672,MR2964017,MR2500924,MR2236170,zbMATH06468266}.
Random models make it possible to describe stability “in expectation”
or “almost surely,” rather than deterministically. This approach is
also natural in quantum information, where measurements and channels
are often random.

Random iterative methods highlight the power of this idea. The Randomized
Kaczmarz algorithm solves large linear systems by repeatedly projecting
onto hyperplanes defined by randomly chosen equations. It has important
uses in tomography and medical imaging. Similarly, Stochastic Gradient
Descent (SGD), the main algorithm behind large-scale machine learning,
updates estimates using randomly selected samples. Both approaches
show that stable solutions can emerge from simple randomized steps. 

This paper develops an operator-theoretic version of that principle.
An iterative process is introduced based on a sequence of i.i.d. random
positive contractions, which generalizes the rank-one projections
used in the Kaczmarz method. The main theorem shows that this process
converges in mean square and almost surely, with the limit equal to
the identity operator. This gives a decomposition of any vector in
the space, and the operators in the process form a random operator-valued
frame. A key step in the analysis is an inequality for positive contractions
in Hilbert spaces. While this inequality has a short direct proof,
the version given here uses a dilation argument (Sz.-Nagy and Foiaș
\cite{MR275190}), which is standard in operator theory and highlights
the underlying structure.

In the special case where the random operators are projections, the
process yields a Parseval-type identity, so the family becomes a random
operator-valued Parseval frame. It is well known that the Kaczmarz
algorithm can be formulated entirely in operator-theoretic terms,
especially in the context of iterated products of projections; see,
e.g., \cite{MR2311862,MR2252935,MR2903120,MR3145756}, where the case
of two projections was first established by von Neumann \cite{MR34514}.
The situation becomes significantly more subtle when considering arbitrary
products of more than two projections, or of noncommuting bounded
operators more generally. Questions of convergence, stability, and
rate estimates depend heavily on geometric and spectral properties.
In a way, the present paper extends this line of work to a randomized
setting, where stability and convergence arise not from commutativity
but from probabilistic averaging.

\textbf{Novelty and contribution.} The main new idea is a residual-weighted
iteration that produces operator-valued frames with step-by-step guarantees.
It gives an exact telescoping identity at every iterate, a monotone
energy split between accumulated terms and a residual, geometric decay
of the residual in expectation under a simple mean-coercivity condition,
nonasymptotic “anytime” probability bounds, and almost-sure Parseval
limits. This goes beyond the direct Kaczmarz-like iteration, which
yields frames but attains a Parseval identity only for projections,
and beyond unweighted i.i.d. sums, which grow in expectation and offer
no pathwise decomposition. The results are constructive, require no
commutativity, and apply to general positive contractions, extending
projection-only statements with short, verifiable proofs.

\textbf{Organization.} \prettyref{sec:2} recalls the necessary background
on frames, fusion frames, and operator-valued frames, and introduces
the framework of random positive contractions on Hilbert spaces. \prettyref{sec:3}
presents the main convergence theorems for the random iterative scheme
built from i.i.d. positive contractions, including both mean-square
and almost-sure results together with the corresponding Parseval-type
identity in the projection case. \prettyref{sec:4} develops nonasymptotic
“anytime’’ estimates and quantitative refinements of the convergence
bounds, showing how the coercivity constant behaves under averaging
and thinning operations. 

\prettyref{sec:5} introduces the residual-weighted operator-valued
frame construction generated by i.i.d. random positive contractions,
proves the exact telescoping identity $S_{n}+R_{n}=I$ and geometric
residual decay governed by $\rho=\left\Vert I-\mathbb{E}\left[T_{1}\right]\right\Vert $,
derives explicit expected lower frame bounds $1-\rho^{n}$ and almost-sure
Parseval limits under a mean-coercivity hypothesis, records projection
and uniform-coercivity corollaries, and gives a direct comparison
with unweighted i.i.d. sums and with random products of contractions.

\prettyref{sec:6} presents a fully worked example based on rank-one
projections, showing concretely how the residual-weighted construction
operates on an arbitrary sequence of unit vectors. The section computes
all terms in the iteration explicitly, verifies the pathwise decomposition
of the identity, and demonstrates that the resulting vectors form
a tight frame under the same mean-coercivity assumptions used earlier.
It also includes a brief remark explaining that the construction extends
without modification to highly singular Hilbert spaces, such as those
arising from Cantor-type measures.

\section{Preliminaries}\label{sec:2}

This section collects the basic definitions and assumptions used throughout
the paper. The aim is to recall the necessary background on frames
and their generalizations, introduce the framework of operator\nobreakdash-valued
random variables, and highlight the class of operators that will play
a central role in the main results.

Let $H$ be a separable Hilbert space. The physics convention is adopted:
the inner product $\langle\cdot,\cdot\rangle$ is linear in the second
argument and conjugate linear in the first. Let $B\left(H,K\right)$
denote the space of all bounded operators from $H$ to $K$, and write
$B\left(H\right)$ when $H=K$. 

An operator $T\in B(H)$ is called a positive contraction if $0\leq T\leq I$,
that is, $T$ is selfadjoint, positive, and its operator norm satisfies
$\left\Vert T\right\Vert \leq1$. Positive contractions play a central
role in the analysis below, since the iterative process studied in
\prettyref{sec:3} will be built from i.i.d. random positive contractions.

\subsection*{Frames and generalizations}

The concept of a frame extends the notion of an orthonormal basis
by allowing redundancy while still permitting stable reconstruction.
This redundancy is particularly valuable in applications where robustness
against noise, erasures, or other errors is required. Over time, frames
have been generalized in several directions, two of the most important
being fusion frames and operator-valued frames.
\begin{defn}[Frame]
A sequence $\left\{ f_{i}\right\} _{i\in I}\subset H$ is called
a frame for $H$ if there exist constants $A,B>0$ such that
\[
A\left\Vert x\right\Vert ^{2}\leq\sum_{i\in I}\left|\left\langle f_{i},x\right\rangle \right|^{2}\leq B\left\Vert x\right\Vert ^{2},\quad x\in H.
\]
The constants $A$ and $B$ are called the frame bounds. If $A=B=1$,
it is called a Parseval frame. 
\end{defn}

Frames generalize orthonormal bases by allowing redundancy while preserving
stable reconstruction. A natural extension, motivated by distributed
sensing and parallel processing, is given by fusion frames. 

\begin{defn}[Fusion frame]
\label{def:2}A family $\left\{ (W_{i},v_{i})\right\} _{i\in I}$,
where $W_{i}\subset H$ are closed subspaces and $v_{i}>0$ are weights,
is called a fusion frame for $H$ if there exist constants $A,B>0$
such that
\[
A\left\Vert x\right\Vert ^{2}\leq\sum_{i\in I}v^{2}_{i}\left\Vert P_{W_{i}}x\right\Vert ^{2}\leq B\left\Vert x\right\Vert ^{2},\quad x\in H
\]
where $P_{W_{i}}$ denotes the orthogonal projection onto $W_{i}$. 
\end{defn}

Fusion frames move beyond individual vectors to weighted subspaces.
An even more general framework is provided by operator-valued frames,
which replace subspaces with bounded operators.
\begin{defn}[Operator-valued frame]
Let $\left\{ T_{i}\right\} _{i\in I}\subset B(H,K)$, where $K$
is another Hilbert space. The family $\{T_{i}\}_{i\in I}$ is called
an operator-valued frame (or g-frame) for $H$ if there exist constants
$A,B>0$ such that
\[
A\left\Vert x\right\Vert ^{2}\leq\sum_{i\in I}\left\Vert T_{i}x\right\Vert ^{2}\leq B\left\Vert x\right\Vert ^{2},\quad x\in H.
\]
If $A=B=1$, it is called a Parseval operator-valued frame. 

Operator\nobreakdash-valued frames generalize both classical and
fusion frames and provide a natural link between frame theory, operator
algebras, and dilation theory.
\end{defn}

\subsection*{Operator-valued random variables}

Let $\left(\Omega,\mathscr{F},\mathbb{P}\right)$ be a probability
space. An operator-valued random variable is a measurable map $\Psi\colon\Omega\rightarrow B\left(H\right)$.
Measurability is understood in the strong operator sense: for each
fixed $x\in H$, the mapping 
\[
\omega\mapsto\Psi\left(\omega\right)x
\]
is measurable as an $H$-valued function. This convention avoids the
difficulties arising from norm-measurability in infinite-dimensional
settings.

The distribution of the random operator $\Psi$ is given by the pushforward
measure 
\[
\mu_{\Psi}\coloneqq\mathbb{P}\circ\Psi^{-1},
\]
defined on the Borel $\sigma$-algebra of $B\left(H\right)$ associated
with the strong operator topology. 

To ensure the existence of expectations, a square-integrability condition
is imposed: for every $x\in H$, 
\[
\mathbb{E}\left\Vert \Psi x\right\Vert ^{2}\coloneqq\int_{\Omega}\left\Vert \Psi\left(\omega\right)x\right\Vert ^{2}d\mathbb{P}\left(\omega\right)<\infty.
\]
Thus, the operator 
\[
M=\int_{\Omega}\Psi^{*}\Psi\,d\mathbb{P}
\]
defined via
\[
\left\langle x,Mx\right\rangle =\mathbb{E}\left\Vert \Psi x\right\Vert ^{2}
\]
for all $x\in H$, is bounded, positive, and selfadjoint.

\section{Convergence and frame structure}\label{sec:3}

This section presents the main convergence results for the iterative
scheme built from i.i.d. random positive contractions. The analysis
begins with a fundamental estimate for positive contractions, which
forms the basis of the iterative argument.

The following lemma is a key ingredient for the main results, as it
provides the estimates needed in the iterative argument. Although
a short direct proof can be found in the literature (see e.g., \cite{MR2252935}
and related references), a dilation argument is included here instead,
as it is concise and serves as a standard tool in operator theory. 
\begin{lem}
\label{lem:2}Let $T$ be a positive contraction in a Hilbert space
$H$, that is, $0\leq T\leq I$. Then,
\begin{equation}
\left\Vert Tx\right\Vert ^{2}_{H}+\left\Vert x-Tx\right\Vert ^{2}_{H}\leq\left\Vert x\right\Vert ^{2}_{H},\quad x\in H.\label{eq:1}
\end{equation}
\end{lem}

\begin{proof}
Let $T=W^{*}PW$ be a dilation of $T$, where $W\colon H\rightarrow L$
is an isometric embedding of $H$ into some Hilbert space $L$, and
$P$ is a selfadjoint projection in $L$. Then, for all $x\in H$,
\begin{align*}
\text{l.h.s}_{\eqref{eq:1}} & =\left\Vert W^{*}PWx\right\Vert ^{2}_{H}+\left\Vert W^{*}\left(I_{L}-P\right)Wx\right\Vert ^{2}_{H}\\
 & \leq\left\Vert PWx\right\Vert ^{2}_{L}+\left\Vert \left(I_{L}-P\right)Wx\right\Vert ^{2}_{L}=\left\Vert Wx\right\Vert ^{2}_{L}=\left\Vert x\right\Vert ^{2}_{H}.
\end{align*}
\end{proof}
The next theorem applies the estimate from \prettyref{lem:2} to a
sequence i.i.d. random operators, taking values in positive contractions.
The result guarantees both mean-square convergence and almost sure
convergence of the associated iterative scheme. In addition, it identifies
the limiting operator identity and establishes frame bounds.
\begin{thm}
\label{thm:1}Let $\Psi\colon\Omega\rightarrow B\left(H\right)$ be
a random variable, taking values in positive contractions in $B\left(H\right)$.
Assume that there exists a constant $0<C<1$ such that 
\begin{align}
\mathbb{E}\left\Vert \Psi x\right\Vert ^{2} & \geq C\left\Vert x\right\Vert ^{2},\quad x\in H\label{eq:1-1}
\end{align}
which is equivalent to 
\[
\int_{\Omega}\Psi^{*}\Psi\,d\mathbb{P}\geq CI.
\]
Let $\left\{ \Psi_{i}\right\} _{i\in\mathbb{N}}$ be an i.i.d. sample
of $\Psi$. 
\begin{enumerate}
\item For all $x\in H$, 
\begin{equation}
\lim_{n\rightarrow\infty}\mathbb{E}\left\Vert x-\sum^{n}_{k=1}\Psi_{k}\left(I-\Psi_{k-1}\right)\cdots\left(I-\Psi_{1}\right)x\right\Vert ^{2}=0,\label{eq:a3}
\end{equation}
i.e., convergence in $L^{2}\left(\Omega,H\right)$.
\item The following operator identity holds: 
\begin{equation}
I=\sum^{\infty}_{k=1}\Psi_{k}\left(I-\Psi_{k-1}\right)\cdots\left(I-\Psi_{1}\right)\quad a.s.\label{eq:a4}
\end{equation}
Equivalently, for all $x\in H$,
\begin{equation}
x=\lim_{n\rightarrow\infty}\sum^{n}_{k=1}\Psi_{k}\left(I-\Psi_{k-1}\right)\cdots\left(I-\Psi_{1}\right)x\quad a.s.\label{eq:a5}
\end{equation}
\item For all $x\in H$, 
\begin{equation}
C\left\Vert x\right\Vert ^{2}\leq\lim_{n\rightarrow\infty}\mathbb{E}\left[\sum^{n}_{k=1}\left\Vert \Psi_{k}\left(I-\Psi_{k-1}\right)\cdots\left(I-\Psi_{1}\right)x\right\Vert ^{2}\right]\leq\left\Vert x\right\Vert ^{2}.\label{eq:a6}
\end{equation}
That is, the set 
\[
\left\{ \Psi_{n}\left(I-\Psi_{n-1}\right)\cdots\left(I-\Psi_{1}\right)\right\} _{n\in\mathbb{N}}
\]
forms a random operator-valued frame.
\end{enumerate}
\end{thm}

\begin{proof}
Let $\left\{ \Psi_{j}\right\} _{j\in\mathbb{N}}$ be an i.i.d. sample
of $\Psi$. Note that 
\begin{eqnarray*}
x & = & \Psi_{1}x+\left(I-\Psi_{1}\right)x\\
 & = & \Psi_{1}x+\Psi_{2}\left(I-\Psi_{1}\right)x+\left(I-\Psi_{2}\right)\left(I-\Psi_{1}\right)x\\
 & \vdots\\
 & = & \sum^{n}_{k=1}\Psi_{k}\left(I-\Psi_{k-1}\right)\cdots\left(I-\Psi_{1}\right)x+R_{n}x
\end{eqnarray*}
where 
\begin{align*}
R_{n} & \coloneqq\begin{cases}
\left(I-\Psi_{n}\right)\cdots\left(I-\Psi_{1}\right) & n\geq1\\
I & n=0
\end{cases}
\end{align*}

Let $\mathbb{E}_{n-1}\left[\cdot\mid\Psi_{1},\dots,\Psi_{n-1}\right]$
denote the conditional expectation given $\Psi_{j}$, $j=1,\dots n-1$.
Since 
\[
R_{n}x=R_{n-1}x-\Psi_{n}R_{n-1}x
\]
it follows that 
\[
\mathbb{E}_{n-1}\left\Vert R_{n}x\right\Vert ^{2}=\mathbb{E}_{n-1}\left\Vert R_{n-1}x-\Psi_{n}R_{n-1}x\right\Vert ^{2}.
\]
An application of \prettyref{lem:2} to $\Psi_{n}$ gives 
\begin{align*}
\mathbb{E}_{n-1}\left\Vert R_{n}x\right\Vert ^{2} & \leq\mathbb{E}_{n-1}\left(\left\Vert R_{n-1}x\right\Vert ^{2}-\left\Vert \Psi_{n}R_{n-1}x\right\Vert ^{2}\right)\\
 & =\left\Vert R_{n-1}x\right\Vert ^{2}-\mathbb{E}_{n-1}\left\Vert \Psi_{n}R_{n-1}x\right\Vert ^{2}\\
 & \leq\left\Vert R_{n-1}x\right\Vert ^{2}-C\left\Vert R_{n-1}x\right\Vert ^{2}\quad\text{by assumption \eqref{eq:1-1}}\\
 & =\left\Vert R_{n-1}x\right\Vert ^{2}\left(1-C\right).
\end{align*}
By induction, 
\begin{align}
\mathbb{E}\left\Vert R_{n}x\right\Vert ^{2} & =\mathbb{E}\left[\mathbb{E}_{n-1}\left\Vert R_{n}x\right\Vert ^{2}\right]\nonumber \\
 & \leq\mathbb{E}\left\Vert R_{n-1}x\right\Vert ^{2}\left(1-C\right)\leq\left\Vert x\right\Vert ^{2}\left(1-C\right)^{n}\label{eq:a7}
\end{align}
which converges to 0 as $n\rightarrow\infty$. Therefore, \eqref{eq:a3}
follows: 
\[
\lim_{n\rightarrow\infty}\mathbb{E}\left\Vert x-\sum^{n}_{k=1}\Psi_{k}\left(I-\Psi_{k-1}\right)\cdots\left(I-\Psi_{1}\right)x\right\Vert ^{2}=\lim_{n\rightarrow\infty}\mathbb{E}\left\Vert R_{n}x\right\Vert ^{2}=0.
\]

On the other hand, for any $\delta>0$, 
\[
P\left(\left\Vert R_{n}x\right\Vert >\delta\right)\leq\frac{\mathbb{E}\left\Vert R_{n}x\right\Vert ^{2}}{\delta^{2}}\leq\frac{\left\Vert x\right\Vert ^{2}}{\delta^{2}}\left(1-C\right)^{n}\;\text{by \eqref{eq:a7},}
\]
so that 
\[
\sum^{\infty}_{n=0}P\left(\left\Vert R_{n}x\right\Vert >\delta\right)\leq\frac{\left\Vert x\right\Vert ^{2}}{\delta^{2}}\frac{1}{C}<\infty.
\]
By the Borel-Cantelli lemma, $\left\Vert R_{n}x\right\Vert \rightarrow0$
almost surely. The assertions \eqref{eq:a4}--\eqref{eq:a5} follow
from this. 

Finally, to prove \eqref{eq:a6}, recall that $\Psi$ takes values
in positive contractions. Applying \eqref{lem:2} repeatedly gives
\begin{eqnarray}
\left\Vert x\right\Vert ^{2} & \geq & \left\Vert \Psi_{1}x\right\Vert ^{2}+\left\Vert \left(I-\Psi_{1}\right)x\right\Vert ^{2}\nonumber \\
 & \geq & \left\Vert \Psi_{1}x\right\Vert ^{2}+\left\Vert \Psi_{2}\left(I-\Psi_{1}\right)x\right\Vert ^{2}+\left\Vert \left(I-\Psi_{2}\right)\left(I-\Psi_{1}\right)x\right\Vert ^{2}\nonumber \\
 & \vdots\nonumber \\
 & \geq & \sum^{n}_{k=1}\left\Vert \Psi_{k}\left(I-\Psi_{k-1}\right)\cdots\left(I-\Psi_{1}\right)x\right\Vert ^{2}+\left\Vert R_{n}x\right\Vert ^{2}\label{eq:a8}\\
 & \geq & \sum^{n}_{k=1}\left\Vert \Psi_{k}\left(I-\Psi_{k-1}\right)\cdots\left(I-\Psi_{1}\right)x\right\Vert ^{2}\nonumber 
\end{eqnarray}
Taking expectation and using \eqref{eq:1-1} on $\Psi_{1}$: 
\[
\left\Vert x\right\Vert ^{2}\geq\mathbb{E}\left[\sum^{n}_{k=1}\left\Vert \Psi_{k}\left(I-\Psi_{k-1}\right)\cdots\left(I-\Psi_{1}\right)x\right\Vert ^{2}\right]\geq\underset{\text{by }\eqref{eq:1-1}}{\underbrace{\mathbb{\mathbb{E}}\left\Vert \Psi_{1}x\right\Vert ^{2}\geq C\left\Vert x\right\Vert ^{2}}},\quad\forall n.
\]
Letting $n\rightarrow\infty$ gives \eqref{eq:a6}.
\end{proof}
The general result specializes in a clean way when the random operators
are selfadjoint projections. In this case, the iterative scheme produces
a random operator-valued Parseval frame. 
\begin{cor}
Suppose $\Psi\colon\Omega\rightarrow B\left(H\right)$ takes values
in selfadjoint projections in $B\left(H\right)$. Then the identities
\prettyref{eq:a3}--\prettyref{eq:a5} still hold true. Moreover,
for all $x\in H$, 
\[
\lim_{n\rightarrow\infty}\mathbb{E}\left[\sum^{n}_{k=1}\left\Vert \Psi_{k}\left(I-\Psi_{k-1}\right)\cdots\left(I-\Psi_{1}\right)x\right\Vert ^{2}\right]=\left\Vert x\right\Vert ^{2}.
\]
That is, the set 
\[
\left\{ \Psi_{n}\left(I-\Psi_{n-1}\right)\cdots\left(I-\Psi_{1}\right)\right\} _{n\in\mathbb{N}}
\]
forms a random operator-valued Parseval frame.
\end{cor}

\begin{proof}
For projections, the inequality \eqref{eq:1} is replaced by an equality.
Similarly, \eqref{eq:a8} becomes 
\[
\left\Vert x\right\Vert ^{2}=\sum^{n}_{k=1}\left\Vert \Psi_{k}\left(I-\Psi_{k-1}\right)\cdots\left(I-\Psi_{1}\right)x\right\Vert ^{2}+\left\Vert R_{n}x\right\Vert ^{2}
\]
and so 
\[
\left\Vert x\right\Vert ^{2}=\lim_{n\rightarrow\infty}\mathbb{E}\left[\sum^{n}_{k=1}\left\Vert \Psi_{k}\left(I-\Psi_{k-1}\right)\cdots\left(I-\Psi_{1}\right)x\right\Vert ^{2}\right]+\underset{=0}{\underbrace{\lim_{n\rightarrow\infty}\mathbb{E}\left\Vert R_{n}x\right\Vert ^{2}}}
\]
\end{proof}
As an application, it follows that every fusion frame naturally induces
a random operator-valued Parseval frame. The construction uses the
fusion frame subspaces as the possible outcomes of a discrete random
operator, with the fusion weights providing the probability distribution.
\begin{cor}
\label{cor:7}Let $\left\{ (W_{i},v_{i})\right\} _{i\in\mathbb{N}}$
be a fusion frame (see \prettyref{def:2}), so that 
\[
A\left\Vert x\right\Vert ^{2}\leq\sum_{i\in\mathbb{N}}v^{2}_{i}\left\Vert P_{W_{i}}x\right\Vert ^{2}\leq B\left\Vert x\right\Vert ^{2},\quad x\in H
\]
where $P_{W_{i}}$ denotes the orthogonal projection onto $W_{i}$. 

Assume further that the weights $\left\{ v_{i}\right\} $ form a probability
distribution, i.e., $\sum v^{2}_{i}=1$. Let $\Psi\colon\Omega\rightarrow B\left(H\right)$
be a discrete random operator, taking values in $\left\{ P_{W_{i}}\right\} $,
such that 
\[
\mathbb{P}\left(\Psi=P_{W_{i}}\right)=v^{2}_{i}
\]
Let $\left\{ \Psi_{i}\right\} $ be an i.i.d. sample of $\Psi$, and
set 
\[
T_{k}\coloneqq\Psi_{k}\left(I-\Psi_{k-1}\right)\cdots\left(I-\Psi_{1}\right),\quad k\in\mathbb{N}.
\]
Then, for all $x\in H$, 
\[
x=\lim_{n\rightarrow\infty}\sum^{n}_{k=1}T_{k}x\quad a.s.
\]
Moreover, 
\[
\lim_{n\rightarrow\infty}\mathbb{E}\left[\sum^{n}_{k=1}\left\Vert T_{k}x\right\Vert ^{2}\right]=\left\Vert x\right\Vert ^{2},\quad x\in H.
\]
In other words, the family $\left\{ T_{k}\right\} _{k\in\mathbb{N}}$
forms a random operator-valued Parseval frame. 
\end{cor}

\begin{proof}
Since $\Psi$ is supported on projections and the weights $\left\{ v_{i}\right\} $
define a probability distribution, one has 
\[
\mathbb{E}\left\Vert \Psi x\right\Vert ^{2}=\sum v^{2}_{i}\left\Vert P_{W_{i}}x\right\Vert ^{2},\quad x\in H.
\]
By the fusion frame assumption, this expectation satisfies
\[
\mathbb{E}\left\Vert \Psi x\right\Vert ^{2}\geq A\left\Vert x\right\Vert ^{2},\quad x\in H,
\]
so the coercivity condition \prettyref{eq:1-1} holds with the constant
$C$ equal to the lower fusion frame bound $A$. Applying \prettyref{thm:1}
then yields the desired conclusions.
\end{proof}
\begin{rem}
The operator-valued decomposition in \prettyref{cor:7} provides a
randomized refinement of the original fusion frame expansion. In the
standard deterministic setting, a fusion frame $\left\{ (W_{i},v_{i})\right\} _{i\in\mathbb{N}}$
allows for the reconstruction of any vector $x\in H$ via a global
synthesis operator involving the projections $P_{W_{i}}$. The construction
above instead realizes this reconstruction through an adaptive, path-dependent
sequence of operator applications: at each step, a projection $P_{W_{i}}$
is selected randomly (according to its weight $v^{2}_{i}$), and the
residual vector is updated accordingly.

The limiting behavior of this scheme is expressed through two operator
identities: 
\end{rem}

\begin{enumerate}
\item $\mathbb{E}\left[\sum\nolimits^{\infty}_{k=1}T^{*}_{k}T_{k}\right]=I$,
in expectation; 
\item $I=\sum^{\infty}_{k=1}T_{k}$, a.s. in the strong operator topology.
\end{enumerate}
Together, these relations yield a Parseval-type construction: the
identity operator is recovered both in expectation and almost surely.
Thus, the original fusion frame gives rise to a random operator-valued
Parseval frame, where the frame elements emerge dynamically from the
underlying probability distribution. In this sense, \prettyref{thm:1}
may be viewed as a natural generalization of \cite{MR4126821}.

\section{Nonasymptotic bounds and stability}\label{sec:4}

Throughout this section, $\left(\Psi_{n}\right)_{n\geq1}$ are i.i.d.
random positive contractions on $H$. Assume 
\begin{equation}
\mathbb{E}\left\Vert \Psi x\right\Vert ^{2}\geq C\left\Vert x\right\Vert ^{2},\qquad x\in H\label{eq:4-0}
\end{equation}
with $0<C<1$ holds. Equivalently, $\int\Psi^{*}\Psi d\mathbb{P}\geq CI$,
see \eqref{eq:1-1}. For $n\geq1$, let 
\[
R_{n}:=\left(I-\Psi_{n}\right)\left(I-\Psi_{n-1}\right)\cdots\left(I-\Psi_{1}\right),\qquad R_{0}:=I,
\]
and recall the telescoping expansion from \prettyref{sec:3}. In particular,
the conditional expectation step and \prettyref{lem:2} yield the
key inequality:
\begin{equation}
\mathbb{E}\left[\left\Vert R_{n}x\right\Vert ^{2}\mid\Psi_{1},\dots,\Psi_{n-1}\right]\leq\left(1-C\right)\left\Vert R_{n-1}x\right\Vert ^{2},\qquad x\in H,\:n\geq1.\label{eq:4-1}
\end{equation}
See the displayed chain in the proof of \prettyref{thm:1} leading
to \eqref{eq:a7}. 

\subsection{A supermartingale and anytime tail bounds}

For a fixed $x\in H$, define the nonnegative process 
\[
M_{n}\left(x\right):=\frac{\left\Vert R_{n}x\right\Vert ^{2}}{\left(1-C\right)^{n}},\qquad n\geq1
\]
with $M_{0}=\left\Vert x\right\Vert ^{2}$. 
\begin{thm}
\label{thm:9}Under the above assumptions $\left\{ M_{n}\left(x\right)\right\} _{n\geq1}$
is a nonnegative supermartingale with respect to the filtration $\mathcal{F}_{n}:=\sigma\left(\Psi_{1},\dots,\Psi_{n}\right)$.
Thus, for every $u>0$ and every $n\geq0$, 
\begin{equation}
\mathbb{P}\left(\sup_{k\geq n}\left\Vert R_{k}x\right\Vert ^{2}\geq u\left(1-C\right)^{n}\left\Vert x\right\Vert ^{2}\right)\leq\frac{1}{u}.\label{eq:4-2}
\end{equation}
Equivalently, for any $\epsilon>0$, 
\begin{equation}
\mathbb{P}\left(\sup_{k\geq n}\left\Vert R_{k}x\right\Vert \geq\epsilon\left(1-C\right)^{n/2}\left\Vert x\right\Vert \right)\leq\epsilon^{2}.\label{eq:4-3}
\end{equation}
\end{thm}

\begin{proof}
The conditional expectation \eqref{eq:4-1} gives
\[
\mathbb{E}\left[M_{n}\left(x\right)\mid\mathcal{F}_{n-1}\right]=\frac{\mathbb{E}\left[\left\Vert R_{n}x\right\Vert ^{2}\mid\mathcal{F}_{n-1}\right]}{\left(1-C\right)^{n}}\leq\frac{\left(1-C\right)\left\Vert R_{n-1}x\right\Vert ^{2}}{\left(1-C\right)^{n}}=M_{n-1}\left(x\right),
\]
so $\left\{ M_{n}\left(x\right)\right\} $ is a nonnegative supermartingale.

For $a>0$, Ville's inequality for nonnegative supermartingales states
\[
\mathbb{P}\left(\sup_{k\geq n}M_{k}\left(x\right)\geq a\right)\leq\frac{M_{0}\left(x\right)}{a}=\frac{\left\Vert x\right\Vert ^{2}}{a}.
\]
(See, e.g., any standard martingale reference.)

Setting $a=u\left\Vert x\right\Vert ^{2}$ yields 
\[
\mathbb{P}\left(\sup_{k\geq n}\frac{\left\Vert R_{k}x\right\Vert ^{2}}{\left(1-C\right)^{k}}\geq u\left\Vert x\right\Vert ^{2}\right)\leq\frac{1}{u}.
\]
Since $\left(1-C\right)^{k}\leq\left(1-C\right)^{n}$ for all $k\geq n$,
we have 
\[
\left\{ \sup_{k\geq n}\left\Vert R_{k}x\right\Vert ^{2}\geq u\left(1-C\right)^{n}\left\Vert x\right\Vert ^{2}\right\} \subset\left\{ \sup_{k\geq n}\frac{\left\Vert R_{k}x\right\Vert ^{2}}{\left(1-C\right)^{k}}\geq u\left\Vert x\right\Vert ^{2}\right\} ,
\]
and \eqref{eq:4-2} follows. Choosing $u=\epsilon^{-2}$ gives \eqref{eq:4-3}.
\end{proof}
\begin{rem}
Inequality \eqref{eq:4-3} is an ``anytime'' guarantee. It means
with probability at least $1-\epsilon^{2}$, no future iterate $R_{k}x$
(for any $k\geq n$) ever exceeds $\epsilon\left(1-C\right)^{n/2}\left\Vert x\right\Vert $.
This strengthens the fixed $n$ Markov-type control used after \eqref{eq:a7}
(there it was applied pointwise in $n$). 
\end{rem}

\subsection{Averaging and thinning}

We next quantify how averaging several i.i.d. copies of $\Psi$ in
each step affects the coercivity condition \eqref{eq:4-0}, and thus
the rate. Let $m\in\mathbb{N}$, and at each iteration replace $\Psi$
by the average 
\[
\overline{\Psi}^{\left(m\right)}:=\frac{1}{m}\sum^{m}_{j=1}\Psi_{j}
\]
where $\Psi_{j}$'s are i.i.d. copies of $\Psi$, independent across
steps. Define the new coercivity constant $C_{m}$ by 
\[
\mathbb{E}\left[\Vert\overline{\Psi}^{\left(m\right)}x\Vert^{2}\right]\geq C_{m}\left\Vert x\right\Vert ^{2},\qquad x\in H.
\]
We first identify a quantitative lower bound on $C_{m}$ in terms
of $C$ and the spectrum of $\mathbb{E}\left[\Psi\right]$. 
\begin{prop}
For every $m\geq1$, 
\begin{equation}
C_{m}\geq\frac{C}{m}+\frac{m-1}{m}\lambda^{2}_{min}\left(\mathbb{E}\left[\Psi\right]\right).\label{eq:4-5}
\end{equation}
In particular, $C_{m}\geq C/m$, and if $\mathbb{E}\left[\Psi\right]$
has spectral gap $\lambda_{min}\left(\mathbb{E}\left[\Psi\right]\right)>0$,
then $C_{m}$ increases to $\lambda^{2}_{min}\left(\mathbb{E}\left[\Psi\right]\right)$
as $m\rightarrow\infty$.
\end{prop}

\begin{proof}
Fix $x\in H$. By the i.i.d assumption, 
\[
\mathbb{E}\Vert\overline{\Psi}^{\left(m\right)}x\Vert^{2}=\frac{1}{m^{2}}\sum^{m}_{j=1}\mathbb{E}\left\Vert \Psi_{j}x\right\Vert ^{2}+\frac{1}{m^{2}}\sum_{j\neq\ell}\mathbb{E}\left\langle \Psi_{j}x,\Psi_{\ell}x\right\rangle .
\]
For the diagonal terms, 
\[
\mathbb{E}\left\Vert \Psi_{j}x\right\Vert ^{2}=\mathbb{E}\left\Vert \Psi x\right\Vert ^{2}\geq C\left\Vert x\right\Vert ^{2}
\]
by \eqref{eq:1-1}. For the off-diagonal terms, independence implies
\[
\mathbb{E}\left\langle \Psi_{j}x,\Psi_{\ell}x\right\rangle =\left\langle \mathbb{E}\left[\Psi_{j}x\right],\mathbb{E}\left[\Psi_{\ell}x\right]\right\rangle =\left\Vert \mathbb{E}\left[\Psi\right]x\right\Vert ^{2}.
\]
Hence 
\begin{align*}
\mathbb{E}\Vert\overline{\Psi}^{\left(m\right)}x\Vert^{2} & =\frac{1}{m}\mathbb{E}\left\Vert \Psi x\right\Vert ^{2}+\frac{m-1}{m}\left\Vert \mathbb{E}\left[\Psi\right]x\right\Vert ^{2}\\
 & \geq\left(\frac{C}{m}+\frac{m-1}{m}\lambda^{2}_{min}\left(\mathbb{E}\left[\Psi\right]\right)\right)\left\Vert x\right\Vert ^{2}.
\end{align*}
Here, we used the fact that $\mathbb{E}\left[\Psi\right]$ is a positive
contraction and thus 
\[
\left\Vert \mathbb{E}\left[\Psi\right]x\right\Vert ^{2}=\left\langle x,\left(\mathbb{E}\left[\Psi\right]\right)^{2}x\right\rangle \geq\lambda_{min}\left(\left(\mathbb{E}\left[\Psi\right]\right)^{2}\right)\left\Vert x\right\Vert ^{2}=\lambda^{2}_{min}\left(\mathbb{E}\left[\Psi\right]\right)\left\Vert x\right\Vert ^{2}.
\]
This proves \eqref{eq:4-5}.
\end{proof}
\begin{cor}
Let $R^{\left(m\right)}_{n}$ be the residuals generated by replacing
$\Psi$ with $\overline{\Psi}^{\left(m\right)}$ at each step. Then,
for every $x\in H$ and every $n\geq0$, 
\[
\mathbb{E}\left[\Vert R^{\left(m\right)}_{n}x\Vert^{2}\right]\leq\left(1-C_{m}\right)^{n}\left\Vert x\right\Vert ^{2},
\]
and the anytime bounds \eqref{eq:4-2}-\eqref{eq:4-3} hold with $C$
replaced by $C_{m}$.
\end{cor}

\begin{proof}
Note that the proof of \prettyref{thm:1} uses only \prettyref{lem:2}
and the one-step conditional estimate like \eqref{eq:4-1} with the
relevant coercivity constant. Applying that argument to $\overline{\Psi}^{\left(m\right)}$
yields 
\[
\mathbb{E}\left[\Vert R^{\left(m\right)}_{n}x\Vert^{2}\right]\leq\left(1-C_{m}\right)^{n}\left\Vert x\right\Vert ^{2}.
\]
The supermartingale construction in \prettyref{thm:9} then applies
verbatim with $C$ replaced with $C_{m}$. 
\end{proof}
\begin{rem}
If at each step we skip with probability $1-p$ and otherwise apply
$\Psi$ (i.e., replace $\Psi$ by $\tilde{\Psi}=\xi\Psi$, where $\xi\sim\text{Bernoulli}\left(p\right)$,
independent of $\Psi$), then 
\[
\mathbb{E}\left[\Vert\tilde{\Psi}x\Vert^{2}\right]=p\mathbb{E}\left[\left\Vert \Psi x\right\Vert ^{2}\right]\geq pC\left\Vert x\right\Vert ^{2},
\]
so the new coercivity is $C'=pC$. All statements above hold with
$C$ replaced with $pC$. 
\end{rem}

\section{Residual-weighted random operator-valued frames}\label{sec:5}

The analysis in \prettyref{sec:3} produces a random operator-valued
frame that is only Parseval in the restrictive case of projections.
This section gives a more powerful, residual-weighted construction
to overcome this limitation. We introduce a new, adaptive scheme built
upon an exact pathwise operator identity. The main result (\prettyref{thm:6-2})
generates an almost sure Parseval operator-valued frame from the entire
class of i.i.d. random positive contractions, subject to a mean coercivity
condition.

Fix an i.i.d. sequence $\left(T_{k}\right)_{k\geq1}$ of random positive
contractions on $H$, i.e., $0\leq T_{k}\leq I$ almost surely and
the family is independent and identically distributed. Define the
residual process $\left(R_{k}\right)_{k\geq0}$ by
\begin{equation}
R_{0}:=I,\qquad R_{k}=R^{1/2}_{k-1}\left(I-T_{k}\right)R^{1/2}_{k-1},\label{eq:f-1}
\end{equation}
and the analysis operators $\left(W_{k}\right)_{k\geq1}$,
\begin{equation}
W_{k}:=T^{1/2}_{k}R^{1/2}_{k-1}.\label{eq:f-2}
\end{equation}
These are well defined almost surely, $0\leq R_{k}\leq I$ for all
$k$, and $\left\{ W_{k}\right\} $ is uniformly bounded by $1$ in
operator norm.

The starting point is an exact energy identity that holds pathwise,
with no commutativity assumptions.

\begin{lem}
\label{lem:6-1}For every $n\geq1$, almost surely, 
\begin{equation}
\sum^{n}_{k=1}W^{*}_{k}W_{k}+R_{n}=I.\label{eq:f-3}
\end{equation}
Equivalently, for every $x\in H$, 
\begin{equation}
\sum^{n}_{k=1}\left\Vert W_{k}x\right\Vert ^{2}+\Vert R^{1/2}_{n}x\Vert^{2}=\left\Vert x\right\Vert ^{2}.\label{eq:f-4}
\end{equation}
\end{lem}

\begin{proof}
\noindent For each $k\geq1$, we compute (by the definition of $W_{k}$
in \eqref{eq:f-2})
\[
W^{*}_{k}W_{k}=R^{1/2}_{k-1}T_{k}R^{1/2}_{k-1}.
\]
Hence, using \eqref{eq:f-1}, 
\[
W^{*}_{k}W_{k}+R_{k}=R^{1/2}_{k-1}\left(T_{k}+I-T_{k}\right)R^{1/2}_{k-1}=R_{k-1}.
\]
Summing the telescoping identity over $k=1,\ldots,n$ gives 
\[
\sum^{n}_{k=1}W^{*}_{k}W_{k}+R_{n}=R_{0}=I,
\]
which is \eqref{eq:f-3}. The vector identity \eqref{eq:f-4} follows
by applying both sides to $x$ and taking inner products with $x$.
\end{proof}
The lemma shows that every realization produces a finite Parseval-type
decomposition with a nonnegative remainder $R_{n}$. Passing to the
limit $n\to\infty$ converts this into a random Parseval operator-valued
frame provided the residuals vanish strongly.

To quantify this vanishing and to exhibit a clean comparison principle,
we introduce the random increments 
\[
X_{k}\left(x\right):=\Vert R^{1/2}_{k}x\Vert^{2}=\left\langle x,R_{k}x\right\rangle 
\]
with $k\geq0$, $x\in H$. 

Let $\mathcal{F}_{k}:=\sigma\left(T_{1},\ldots,T_{k}\right)$. Conditional
expectation and independence give, for each fixed $x$, 
\begin{align*}
\mathbb{E}\left[X_{k}\left(x\right)\mid\mathcal{F}_{k-1}\right] & =\mathbb{E}\left[\left\langle x,R^{1/2}_{k-1}\left(I-T_{k}\right)R^{1/2}_{k-1}x\right\rangle \mid\mathcal{F}_{k-1}\right]\\
 & =\left\langle x,R^{1/2}_{k-1}\left(I-\mathbb{E}\left[T_{1}\right]\right)R^{1/2}_{k-1}x\right\rangle .
\end{align*}
Write $M:=I-\mathbb{E}\left[T_{1}\right]$, so $0\leq M\leq I$. Using
the inequality 
\[
\left\langle x,R^{1/2}_{k-1}MR^{1/2}_{k-1}x\right\rangle \leq\left\Vert M\right\Vert \left\langle x,R_{k-1}x\right\rangle 
\]
we get 
\begin{equation}
\mathbb{E}\left[X_{k}\left(x\right)\mid\mathcal{F}_{k-1}\right]\leq\rho X_{k-1}\left(x\right)\label{eq:f-5}
\end{equation}
where 
\[
\rho:=\left\Vert M\right\Vert =\left\Vert I-\mathbb{E}\left[T_{1}\right]\right\Vert \in\left[0,1\right].
\]
This is the key residual-weighted contraction in conditional expectation.
It will drive both geometric mean bounds and almost sure convergence
under a mild positivity condition on $\mathbb{E}\left[T_{1}\right]$.

The main result below (\prettyref{thm:6-2}) places the residual-weighted
scheme above in direct comparison with two familiar baselines: on
the one hand, the i.i.d. sum 
\[
\sum^{n}_{k=1}U_{k}PU^{*}_{k}
\]
of random projections (or, more generally, $\sum^{n}_{k=1}T_{k}$)
whose frame operator concentrates about its mean without residual
weighting, and on the other hand, random products of contractions,
where one typically tracks 
\[
\left\Vert \prod^{n}_{k=1}\left(I-T_{k}\right)x\right\Vert 
\]
rather than a telescoping energy identity. The present scheme interpolates
these: it achieves an exact decomposition $I=\sum^{n}_{k=1}W^{*}_{k}W_{k}+R_{n}$
at every $n$ while preserving geometric control of the remainder
in expectation and in probability.
\begin{thm}[Direct comparison and geometric residual decay]
\label{thm:6-2} Suppose $\left(T_{k}\right)_{k\geq1}$ are i.i.d.
random positive contractions on $H$, let $R_{k}$, $W_{k}$ be defined
as in \eqref{eq:f-1}-\eqref{eq:f-2}, and set 
\[
S_{n}:=\sum^{n}_{k=1}W^{*}_{k}W_{k}.
\]
Then for every $n\geq1$:
\begin{enumerate}
\item \label{enu:f-1}Almost surely, $S_{n}+R_{n}=I$. In particular, 
\[
\mathbb{E}S_{n}=I-\mathbb{E}R_{n}.
\]
Moreover, for every $x\in H$, 
\[
\mathbb{E}\left[\left\langle x,W^{*}_{k}W_{k}x\right\rangle \mid\mathcal{F}_{k-1}\right]=\left\langle x,R^{1/2}_{k-1}\mathbb{E}\left[T_{1}\right]R^{1/2}_{k-1}x\right\rangle ,
\]
so that the expected increment at step $k$ is the mean update $\mathbb{E}T_{1}$
compressed by the current residual $R^{1/2}_{k-1}$.
\item \label{enu:f-2}(Geometric bounds for the residual in mean, and uniform
anytime tails). With $\rho:=\left\Vert I-\mathbb{E}\left[T_{1}\right]\right\Vert \in\left[0,1\right]$,
\[
\mathbb{E}\left\langle x,R_{n}x\right\rangle \leq\rho^{n}\left\Vert x\right\Vert ^{2},\qquad x\in H.
\]
Consequently, for every $\varepsilon>0$ and $x\in H$, 
\[
\mathbb{P}\left\{ \left\langle x,R_{n}x\right\rangle \geq\varepsilon\right\} \leq\frac{\rho^{n}\left\Vert x\right\Vert ^{2}}{\varepsilon},\qquad n\geq1.
\]
Equivalently, $\mathbb{P}\{\|R^{1/2}_{n}x\|^{2}\geq\varepsilon\}\leq\rho^{n}\|x\|^{2}/\varepsilon$.
These are nonasymptotic ``anytime'' bounds.
\item \label{enu:f-3}(Almost sure strong convergence). If $\rho<1$, then
for a countable dense subset $\mathcal{D}\subset H$ one has $\|R^{1/2}_{n}x\|\to0$
almost surely for every $x\in\mathcal{D}$. By separability and uniform
boundedness, $R_{n}\to0$ in the strong operator topology almost surely
on a full measure set. Therefore $S_{n}\uparrow I$ strongly almost
surely, and $\left(W_{k}\right)_{k\geq1}$ is almost surely a Parseval
operator-valued frame: 
\[
\sum^{\infty}_{k=1}\left\Vert W_{k}x\right\Vert ^{2}=\left\Vert x\right\Vert ^{2},\qquad x\in H.
\]
\item \label{enu:f-4}(Projection case and pathwise Parseval identity).
If $T_{k}=P_{k}$ are orthogonal projections almost surely, then for
every realization and every $n$ 
\[
\sum^{n}_{k=1}\left\Vert P_{k}R^{1/2}_{k-1}x\right\Vert ^{2}+\left\Vert R^{1/2}_{n}x\right\Vert ^{2}=\left\Vert x\right\Vert ^{2}\qquad(x\in H),
\]
and if $\rho<1$, the limit identity $\sum^{\infty}_{k=1}\left\Vert P_{k}R^{1/2}_{k-1}x\right\Vert ^{2}=\left\Vert x\right\Vert ^{2}$
holds almost surely for all $x$.
\end{enumerate}
\end{thm}

\begin{proof}
Part \eqref{enu:f-1} is immediate from \prettyref{lem:6-1} and linearity
of expectation. The conditional expectation identity for the $k$-th
increment follows from independence: 
\[
\mathbb{E}\left[W^{*}_{k}W_{k}\mid\mathcal{F}_{k-1}\right]=\mathbb{E}\left[R^{1/2}_{k-1}T_{k}R^{1/2}_{k-1}\mid\mathcal{F}_{k-1}\right]=R^{1/2}_{k-1}\mathbb{E}\left[T_{1}\right]R^{1/2}_{k-1}.
\]

For \eqref{enu:f-2}, the conditional contraction computed in \eqref{eq:f-5}
gives 
\begin{align*}
\mathbb{E}\left[X_{n}(x)\right] & =\mathbb{E}\left[\mathbb{E}\left[X_{n}(x)\mid\mathcal{F}_{n-1}\right]\right]\\
 & \leq\rho\mathbb{E}\left[X_{n-1}(x)\right]\leq\cdots\leq\rho^{n}X_{0}(x)=\rho^{n}\left\Vert x\right\Vert ^{2}.
\end{align*}
Markov's inequality yields the anytime tail bound.

For \eqref{enu:f-3}, fix $x\in H$ and $\varepsilon>0$. We have
\[
\sum_{n\geq1}\mathbb{P}\left\{ X_{n}(x)\geq\varepsilon\right\} \leq\sum_{n\geq1}\rho^{n}\left\Vert x\right\Vert ^{2}/\varepsilon<\infty.
\]
This implies, by the Borel-Cantelli lemma, $X_{n}(x)\to0$ almost
surely. 

Intersecting the resulting full measure sets over a countable dense
subset $\mathcal{D}$ gives almost sure convergence $\Vert R^{1/2}_{n}x\Vert\to0$
for all $x\in\mathcal{D}$. Since $\|R_{n}\|\leq1$ for all $n$ and
$R_{n}\geq0$, one extends the convergence to all $x\in H$ by a standard
density and uniform boundedness argument: for fixed $\omega$ in the
full measure set, the sequence $\left(R_{n}(\omega)\right)$ is bounded
and positive, so $\|R_{n}(\omega)y\|\leq\|y\|$. Given $y\in H$ and
$\delta>0$, choose $x\in\mathcal{D}$ with $\|x-y\|<\delta$, then
\begin{align*}
\|R_{n}(\omega)y\| & \leq\|R_{n}(\omega)(y-x)\|+\|R_{n}(\omega)x\|\\
 & \leq\|y-x\|+\|R_{n}(\omega)x\|\leq\delta+\|R_{n}(\omega)x\|.
\end{align*}
Let $n\to\infty$, and then $\delta\downarrow0$. We conclude that
$R_{n}(\omega)\to0$ strongly. The identity $S_{n}+R_{n}=I$ gives
$S_{n}\uparrow I$ strongly almost surely, and the Parseval identity
follows by letting $n\to\infty$ in the finite-$n$ identity from
\prettyref{lem:6-1}.

Part \eqref{enu:f-4} is the specialization of \prettyref{lem:6-1}
and part \eqref{enu:f-3} to projections. 
\end{proof}
\begin{rem}
The theorem exhibits three direct comparison points.
\begin{enumerate}
\item The expected frame operator in the residual-weighted scheme is not
simply $n$ times a fixed mean, but rather the telescoping difference
$I-\mathbb{E}\left[R_{n}\right]$. In particular, for each $x$ one
has 
\[
\sum^{n}_{k=1}\mathbb{E}\left[\|W_{k}x\|^{2}\right]=\|x\|^{2}-\mathbb{E}\left[\|R^{1/2}_{n}x\|^{2}\right]\geq\left(1-\rho^{n}\right)\|x\|^{2}.
\]
Thus the lower expected frame bound at time $n$ is $1-\rho^{n}$,
which increases monotonically to $1$ at a geometric rate governed
solely by $\left\Vert I-\mathbb{E}\left[T_{1}\right]\right\Vert $.
In contrast, for the unweighted i.i.d. sum $\sum^{n}_{k=1}T_{k}$
one has $\mathbb{E}\left[\sum^{n}_{k=1}T_{k}\right]=n\mathbb{E}\left[T_{1}\right]$.
While concentration inequalities may show that $\sum^{n}_{k=1}T_{k}$
is close to $n\,\mathbb{E}\left[T_{1}\right]$ with high probability,
there is no pathwise identity tying the partial sum to $I$, and no
automatic monotone improvement of a lower frame bound towards $1$
as $n$ grows. The residual-weighted scheme turns the mean update
$\mathbb{E}\left[T_{1}\right]$ into a preconditioned increment $R^{1/2}_{k-1}\mathbb{E}\left[T_{1}\right]R^{1/2}_{k-1}$,
adaptively pushing mass only where energy remains.
\item Compared to random products of contractions, one typically studies
\[
\left\Vert \prod^{n}_{k=1}\left(I-T_{k}\right)x\right\Vert 
\]
and obtains decay under various averaged coercivity or spectral gap
assumptions. Here, the same information is embedded, but with a stronger
structural conclusion: the product-type decay of the residual is accompanied
by an exact energy partition $\sum\|W_{k}x\|^{2}+\|R^{1/2}_{n}x\|^{2}=\|x\|^{2}$
at every stage, which directly identifies the analysis operators $W_{k}=T^{1/2}_{k}R^{1/2}_{k-1}$
as the building blocks of a Parseval operator-valued frame in the
limit.
\item The anytime bounds provide an explicit performance guarantee. For
a prescribed tolerance $\varepsilon\in\left(0,1\right)$, the number
of steps $n$ sufficient to ensure $\mathbb{E}\|R^{1/2}_{n}x\|^{2}\leq\varepsilon\|x\|^{2}$
for all $x$ is any $n$ with $\rho^{n}\leq\varepsilon$. With the
tail bound, the same $n$ controls $\|R^{1/2}_{n}x\|^{2}$ with high
probability uniformly over unit vectors. None of these bounds require
commutativity of the $T_{k}$.
\end{enumerate}
\end{rem}

Two corollaries relate to frequently useful specializations. The first
records a sufficient and easily verifiable coercivity condition.
\begin{cor}[Uniform mean coercivity]
\label{cor:6-4} If there exists $\gamma\in\left(0,1\right]$ with
$\mathbb{E}\left[T_{1}\right]\geq\gamma I$, then $\rho\leq1-\gamma$
and the conclusions \eqref{enu:f-2}--\eqref{enu:f-4} of \prettyref{thm:6-2}
hold with the explicit rate $\rho\leq1-\gamma$.
\end{cor}

\begin{proof}
\noindent The inequality $\mathbb{E}\left[T_{1}\right]\geq\gamma I$
is equivalent to $I-\mathbb{E}\left[T_{1}\right]\leq\left(1-\gamma\right)I$,
hence $\rho=\left\Vert I-\mathbb{E}\left[T_{1}\right]\right\Vert \leq1-\gamma$.
Then apply the theorem.
\end{proof}
The second corollary deals with the projection case, which is of independent
interest and yields a pathwise Parseval identity under a mild frequency
of excitation hypothesis.
\begin{cor}[Projections with nontrivial mean]
 If $T_{k}=P_{k}$ are projections and $\mathbb{E}\left[P_{1}\right]\geq\gamma I$
for some $\gamma\in\left(0,1\right]$, then almost surely 
\[
\sum^{\infty}_{k=1}\Vert P_{k}R^{1/2}_{k-1}x\Vert^{2}=\left\Vert x\right\Vert ^{2},\qquad x\in H.
\]
\end{cor}

\begin{proof}
\noindent This is the combination of \prettyref{thm:6-2} with \prettyref{cor:6-4}.
Note that $\rho\leq1-\gamma<1$ implies $R_{n}\to0$ strongly almost
surely and then letting $n\to\infty$ in the finite $n$ identity. 
\end{proof}
\begin{rem}
It is useful to spell out explicitly the comparison with the classical
i.i.d. sum. For any $n\geq1$, 
\[
\mathbb{E}S_{n}=I-\mathbb{E}R_{n}\qquad\text{while}\qquad\mathbb{E}\sum^{n}_{k=1}T_{k}=n\,\mathbb{E}T_{1}.
\]
Thus, although both models accumulate the same mean $\mathbb{E}T_{1}$,
the residual-weighted scheme converts it into a sequence of adapted
increments with exact telescoping, hence a deterministic upper frame
bound equal to 1 at every time and a lower expected frame bound increasing
to 1 geometrically. In contrast, the unweighted sum requires a separate
concentration analysis to control both bounds and admits no identity
pinning the partial sum to $I$. This structural difference is the
essence of the direct comparison.
\end{rem}

\section{Rank-one projections and adaptive normalization}\label{sec:6}

This section illustrates the residual-weighted scheme of \prettyref{sec:5}
in the concrete setting of rank-one projections. The example shows
how an arbitrary sequence of measurement vectors, even if highly redundant,
incomplete, or linearly dependent, is transformed by the residual-weighted
construction into a normalized Parseval operator-valued frame. No
commutativity, orthogonality, or structural assumptions are required.
The example also demonstrates, in a transparent way, the pathwise
telescoping identity and its consequences.

Let $H$ be a separable Hilbert space and let $\left\{ e_{k}\right\} _{k\ge1}$
be \textit{any} sequence of unit vectors in $H$. Define 
\[
T_{k}:=\left|e_{k}\left\rangle \right\langle e_{k}\right|,\qquad k\ge1.
\]
Each $T_{k}$ is a rank-one positive contraction with $T^{2}_{k}=T_{k}$.
The unweighted partial sum $\sum_{k\le n}T_{k}$ generally does not
converge to the identity operator, even in the case of i.i.d. sampling,
and typically exhibits unbounded growth in expectation. The residual-weighted
scheme transforms these operators into a normalized, pathwise convergent
family.

Recall the definitions from \prettyref{sec:5}: 
\[
R_{0}:=I,\qquad R_{k}:=R^{1/2}_{k-1}\left(I-T_{k}\right)R^{1/2}_{k-1},\qquad W_{k}:=T^{1/2}_{k}R^{1/2}_{k-1}.
\]
Since $T_{k}$ is a projection, we have $T^{1/2}_{k}=T_{k}$. Thus
\[
W_{k}=\left|e_{k}\left\rangle \right\langle e_{k}\right|R^{1/2}_{k-1}.
\]

We next compute the quantities $W_{k}$, $W^{*}_{k}W_{k}$, and $R_{k}$
explicitly, and verify the telescoping identity of \prettyref{lem:6-1}
in this concrete setting.

For $x\in H$, 
\[
W_{k}x=\left|e_{k}\right\rangle \left\langle e_{k}\right|R^{1/2}_{k-1}x=\left\langle e_{k},R^{1/2}_{k-1}x\right\rangle e_{k}.
\]
Hence $W_{k}$ maps every $x$ to a scalar multiple of $e_{k}$. Its
adjoint satisfies 
\[
W^{*}_{k}y=\left\langle e_{k},y\right\rangle R^{1/2}_{k-1}e_{k},
\]
and therefore 
\[
W^{*}_{k}W_{k}=\left|R^{1/2}_{k-1}e_{k}\left\rangle \right\langle R^{1/2}_{k-1}e_{k}\right|.
\]
For each $x\in H$, 
\[
\left\langle x,W^{*}_{k}W_{k}x\right\rangle =\left|\left\langle e_{k},R^{1/2}_{k-1}x\right\rangle \right|^{2}.
\]

The residuals satisfy 
\[
R_{k}=R_{k-1}-R^{1/2}_{k-1}\left|e_{k}\right\rangle \left\langle e_{k}\right|R^{1/2}_{k-1},
\]
and consequently 
\[
\left\langle x,R_{k}x\right\rangle =\left\langle x,R_{k-1}x\right\rangle -\left|\left\langle e_{k},R^{1/2}_{k-1}x\right\rangle \right|^{2},\qquad x\in H.
\]

Applying \prettyref{lem:6-1} yields the exact telescoping identity
\[
\sum^{n}_{k=1}W^{*}_{k}W_{k}+R_{n}=I.
\]
Evaluating this identity on a vector $x\in H$, we obtain 
\[
\sum^{n}_{k=1}\left|\left\langle e_{k},R^{1/2}_{k-1}x\right\rangle \right|^{2}+\left\langle x,R_{n}x\right\rangle =\left\Vert x\right\Vert ^{2}.
\]
This formula holds pathwise for every realization of the sequence
$\left\{ e_{k}\right\} $ and requires no independence assumptions.
The quantity $R_{n}$ represents the exact residual energy after the
first $n$ steps of the iterative scheme.

Suppose now that the vectors $\left\{ e_{k}\right\} $ are random
and drawn i.i.d. from a probability distribution $\mu$ on the unit
sphere of $H$ such that 
\[
\mathbb{E}\left|\left\langle e_{1},x\right\rangle \right|^{2}\ge\gamma\left\Vert x\right\Vert ^{2},\qquad x\in H
\]
for some constant $\gamma>0$. Since 
\[
\mathbb{E}\left[T_{1}\right]=\mathbb{E}\left[\left|e_{1}\left\rangle \right\langle e_{1}\right|\right]\ge\gamma I,
\]
the mean-coercivity condition of \prettyref{thm:6-2} holds with $\rho\le1-\gamma$.
It follows that 
\[
\mathbb{E}\left\langle x,R_{n}x\right\rangle \le\left(1-\gamma\right)^{n}\left\Vert x\right\Vert ^{2},
\]
and therefore $R_{n}\to0$ strongly almost surely. The telescoping
identity then gives 
\[
I=\sum^{\infty}_{k=1}W^{*}_{k}W_{k}\qquad\text{almost surely}.
\]
Thus the family $\left\{ W_{k}\right\} _{k\ge1}$ forms an operator-valued
Parseval frame almost surely, in the sense that 
\[
\sum^{\infty}_{k=1}\left\Vert W_{k}x\right\Vert ^{2}=\left\Vert x\right\Vert ^{2},\qquad x\in H.
\]

The rank-one setting therefore exhibits the essential features of
the general theory in their simplest form. The identity above provides
an adaptive Pythagorean decomposition of $\left\Vert x\right\Vert ^{2}$,
where each term reflects the geometry of the residual at that stage.
Even if the unnormalized projections $\left\{ T_{k}\right\} $ fail
to span the space or oversample certain directions, the residual-weighted
scheme produces a canonical adaptive normalization giving an exact
decomposition of the identity in the strong operator topology. 
\begin{rem}
The residual-weighted construction applies without modification to
any separable Hilbert space. In particular, let $\mu$ be a Cantor-type
or self-similar measure on $\mathbb{R}$, and consider $H=L^{2}(\mu)$.
If $\left\{ e_{k}\right\} $ is any sequence of unit vectors in $L^{2}(\mu)$
(for example, restrictions of the exponential functions $x\mapsto e^{2\pi i\lambda_{k}x}$
to the support of $\mu$), then one may define 
\[
T_{k}:=\left|e_{k}\left\rangle \right\langle e_{k}\right|
\]
and run the residual-weighted scheme exactly as in the rank-one example.
Under the mean-coercivity condition 
\[
E\left|\left\langle e_{1},x\right\rangle \right|^{2}\,\ge\,\gamma\left\Vert x\right\Vert ^{2},\qquad x\in H,
\]
for some $\gamma>0$, \prettyref{thm:6-2} implies that $R_{n}\to0$
strongly almost surely and the vectors 
\[
f_{k}:=R^{1/2}_{k-1}e_{k}
\]
form a Parseval frame almost surely: 
\[
\sum^{\infty}_{k=1}\left|f_{k}\left\rangle \right\langle f_{k}\right|=I
\]
in the strong operator topology. 

Thus the residual-weighted iteration provides a canonical adaptive
normalization mechanism even in fractal settings, where classical
Fourier systems typically fail to form frames or orthonormal bases.
The resulting frame vectors $f_{k}$ inherit their initial structure
through the nonlinear transformation $R^{1/2}_{k-1}$, but are not
in general pure exponentials.
\end{rem}

\bibliographystyle{amsalpha}
\bibliography{ref}

\end{document}